\documentclass[reqno,10pt]{amsart}
\usepackage{amssymb}
\usepackage{latexsym}
\usepackage{hyperref}
\usepackage{amsmath}
\usepackage{amscd}
\usepackage{color}
\usepackage{enumerate}
\usepackage{amsfonts}
\usepackage{graphicx}
\usepackage{mathrsfs}
\usepackage{pifont}
\usepackage{setspace}

\usepackage{tabu}

\theoremstyle{plain}
\newtheorem{theorem}{Theorem}

\newtheorem{lemma}[theorem]{Lemma}

\theoremstyle{definition}

\newtheorem{remark}[theorem]{Remark}

\numberwithin{equation}{section}
\numberwithin{theorem}{section}

\def\geq{\geqslant}
\def\leq{\leqslant}

\hypersetup{colorlinks = true,linkcolor = blue,anchorcolor =red,citecolor = blue,filecolor = red,urlcolor = red,
            pdfauthor=author}

\begin{document}

\title[dispersive equations]
{Periodic Schr\"odinger map flow on  K\"ahler manifolds }	

\author[S. Wang]{Sheng Wang}
\address{Shanghai Center for Mathematical Sciences, Fudan University, Shanghai 200433, P.R. China}
\email{19110840011@fudan.edu.cn}

\author[Y. Zhou]{Yi Zhou}
\address{School of Mathematical Sciences, Fudan University, Shanghai 200433, P.R. China}
\email{yizhou@fudan.edu.cn}

\subjclass[2010]{Primary: 35Q55; Secondary: 35B44}
\date{\today}

\keywords{Wei-Yue Ding's Conjecture;  K\"ahler manifolds; Div-Curl Lemma; Gobal Well-posedness.}

\maketitle

\begin{abstract}

Wei-Yue Ding \cite{Ding 2002} proposeed a proposition about Schr\"odinger map flow in 2002 International Congress of Mathematicians in Beijing, which is called Wei-Yue Ding conjecture by Rodnianski-Rubinstein-Staffilani \cite{Rodnianski 2009}.
They proved \cite{Rodnianski 2009} that Schr\"odinger map flow for maps from the real line into  K\"ahler manifolds and for maps from the circle into Riemann surfaces is globally well-posed  which is the first significant advance in this conjecture by translating the Schr\"odinger map flow into  nonlinear Schr\"odinger-type  equations or (systems) and partially solved this conjecture.  In this article, we will derive a new div-curl type lemma and combined it with energy and ``momentum" balance law to get some space-time estimates.
Based on this, we prove the Schr\"odinger map flow for maps from the circle into K\"ahler manifolds is globally regular. So far, the Wei-Yue Ding's conjecture has been completely solved.


\end{abstract}

\section {Introduction}

Let $(M,g)$ be an  $m$-dimensional complete Riemannian manifold,
and let
$(N,\omega,J,h)$ be a  $2n$-dimensional complete symplectic manifold  with a compatible almost
complex structure $J$, which means $(J\cdot,J\cdot)=(\cdot,\cdot)$ and 
$h(\cdot,\cdot)=\omega(\cdot,J\cdot)$. This defines a complete Riemannian metric on $N$.

We introduce the Fr\'echet manifold $X:=C^\infty(M,N)$ which is the space of all smooth maps from $M$ to $N$,
endowed with a symplectic structure,
\begin{equation}
	\Omega(V,W)|_u=\int_M u^{\star}\omega(V,W) d V_{M,g}, \quad \forall\; V,W\in T_uX=\Gamma(M,u^{\star}TN),
\end{equation}
where the tangent space to $X$ at a map $u:M\rightarrow N$ is the space of
smooth sections of $u^{\star}TN \rightarrow M$ and where $dV_{M,g}$ denotes the
volume form on $M$ induced by Riemannian metric $g$. The form $\Omega$ is
non-degenerate, i.e., endows $X$ with an injective map $TX\rightarrow
T^{\star} X$.

The energy function on $X$ can be represented by
\begin{equation}\label{Enr}
\mathcal{E}(u)=\frac{1}{2} \int_M |du|_{g^\sharp\otimes u^{\star}h}^2 dV_{M,g},
\end{equation}
where we denote by $g^\sharp$ the metric induced by $g$ on
$T^{\star} M$ and where we view $du$ as a section of $T^{\star}
M \otimes u^{\star}TN \rightarrow M$ and equip this bundle with the metric
$g^\sharp \otimes u^{\star}h$.

As the similar arguments in  \cite{Rodnianski 2009}, by Hamilton's principle, we can obtain the corresponding Hamiltonian flow:
\begin{equation}
\left\{
\begin{aligned}
& \partial_t u =-J\tau(u),\\
& u|_{\{0\}\times M}=u_0,
\end{aligned}
\right.
\label{schf}
\end{equation}
on $\left(X, \Omega\right)$ from \eqref{Enr}, where
$\tau(u):=tr_{g^\sharp} \nabla du$ is called the tension field of $u$ and $\nabla$ is the connection
on $T^\star M\otimes u^{\star}TN \rightarrow M$ induced from the Levi-Civita
connection on $(M,g)$ and the pulled-back one from $(N,h)$.

In fact, the Hamiltonian flow \eqref{schf} is also called  Schr\"odinger map flow, introduced by Ding-Wang in \cite{Ding 1998}.  

An interesting problem  about this flow \eqref{schf} is the existence and uniqueness. But if we consider the general case, we may not have the global existence or well-posedness because this flow is defined on infinite-dimensional spaces. So we usually restrict to the K\"ahler case. Ding-Wang \cite{Ding 2001} proved the local well-posedness, and then McGahagan \cite{McGahagan 2007} gave the similar results with another way by parabolic approximation. This result can be described as follows:
\begin{theorem}[Local well-posedness]
Suppose $\left(M, g\right)$ is an $m$-dimensional  complete Riemannian manifold and $\left(N, J, h\right)$ is a complete K\"ahler manifold with bounded geometry. If the initial data $u_0\in H^k\left(M, N\right)$ with $k>\frac{m}{2}+1$, then the system \eqref{schf} admits a unique  flow $u\in C^0 \left(\left[0, T\right], H^k\left(M, N\right) \right)$ where $T<T_0$ and $T_0$ only depends on the geometry of the manifold $N$ and $\left \Vert \nabla u_0\right\Vert_{H^{\left[m/2\right]+1}}$. In addition, we have the following estimates:
\begin{equation*}
\left \Vert \nabla u\right\Vert_{H^{\left[m/2\right]+1}} \leq \frac{C_1}{\left(T_0-t\right)^{C_2}} \quad \forall t \in\left[0, T_0\right),
\end{equation*}
where the constants $C_1$ and $C_2$  only depend on  the geometry of the manifold $N$ and $\left \Vert \nabla u_0\right\Vert_{H^{\left[m/2\right]+1}}$. Particularly, when the initial data $u_0\in H^k\left(M, N\right)$ with any $k\geq2$, then the flow $u\in C^{\infty}\left(\left[0,T\right]\times M, N\right)$.
\end{theorem}

We recall the  Wei-Yue Ding's conjecture:
\subsection{Conjecture \cite{Ding 2002}} 
 
The Schr\"odinger map flow is globally well-posed for maps from one-dimensional domains into compact K\"ahler manifolds.

Rodnianski-Rubinstein-Staffilani \cite{Rodnianski 2009} first studied this conjecture, and have taken an important step forward in the research of the conjecture. They proved that the Schr\"odinger map flow for maps from the real line into  K\"ahler manifolds and for maps from the circle into Riemann surfaces is globally well-posed. These results can be described as follows:

\begin{theorem}
	Suppose $\left(M, g\right)= \left(\mathbb{R}, dx \otimes dx \right)$ and $\left(N, J, h\right)$ is a complete K\"ahler manifold with bounded geometry. If the initial data $u_0\in H^k\left(\mathbb{R}, N\right)$ with $k\geq 2$, then the system \eqref{schf} admits a unique  flow $u\in C^0 \left(\left[0, \infty\right), H^k\left(\mathbb{R}, N\right) \right)$.  Particularly, when the initial data $u_0\in H^k\left(\mathbb{R}, N\right)$ with any $k\geq2$, then the flow $u\in C^{\infty}\left(\left[0,\infty \right)\times \mathbb{R}, N\right)$. 
\end{theorem}

\begin{theorem}
	Suppose $\left(M, g\right)= \left(\mathbb{S}^1, dx \otimes dx \right)$ and $\left(N, J, h\right)$ is a complete Riemann surface with bounded geometry. If the initial data $u_0\in H^k\left(\mathbb{S}^1, N\right)$ with $k\geq 2$, then the system \eqref{schf} admits a unique  flow $u\in C^0 \left(\left[0, \infty\right), H^k\left(\mathbb{S}^1, N\right) \right)$.  Particularly, when the initial data $u_0\in H^k\left(\mathbb{S}^1, N\right)$ with any $k\geq2$, then the flow $u\in C^{\infty}\left(\left[0,\infty \right)\times \mathbb{S}^1, N\right)$. 
\end{theorem}

In the Rodnianski-Rubinstein-Staffilani work \cite{Rodnianski 2009}, they need to translate the Schr\"odinger map flow into  nonlinear Schr\"odinger-type  equations or (systems). For the  maps from real line into K\"ahler manifolds, they can find a set of global parallel frams to translate the Schr\"odinger map flow into  nonlinear Schr\"odinger-type  equations or (systems) and use the Strichartz estimates to obatin the $L^4$-norm priori estimates to establish the global well-posedness. For the maps from circle into Riemann surfaces, they use the holonomy method and some transformation to translate the Schr\"odinger map flow into cubic nonlinear Schr\"odinger and adapt the Bourgain's results \cite{Bourgain 1993} to get a priori estimates for $L^4$-norm to establish the global well-posedness. 

For the maps from circle into  K\"ahler manifolds, it is unable to translate the Schr\"odinger map flow into  nonlinear Schr\"odinger-type  equations or (systems) so that the classical methods for studying the dispersive equation (such as Strichartz estimates) are no longer suitable for the research of Schr\"odinger map flow.  So in this work, we study the Schr\"odinger map flow from a new perspective. We don't need to translate the Schr\"odinger map flow into  nonlinear Schr\"odinger-type  equations or (systems), and  don't need to use the Strichartz estimates. Just like our  previous work \cite{wang 2022} for time-like extremal hypersurface equations, we also derive a new div-curl type lemma \ref{dc}  that connects energy and ``momentum" in the periodic case.  In our method, we can use the new div-curl type lemma \ref{dc} to study some balance law (i,e. $\partial_t f_1 +  \partial_x f_2 = f^{re}$).  Particularly, we can use the new div-curl lemma \ref{dc} to fully explore the energy balance law and ``momentum" balance law. Interestingly, in our approch, the ``momentum" balance law is as important as the energy balance law, this gives a new perspective in the study of the estimates of the nonlinearity.  Based  on this and making full use of the almost complex structure $J$ in K\"ahler manidolds, we can prove that the Schr\"odinger map flow for maps from the  $\mathbb{S}^1$ into  K\"ahler manifolds is globally well-posed. So far, the Wei-Yue Ding's conjecture has been completely proved.

Now, we describe our main results as follow:

\begin{theorem}
	Suppose $\left(M, g\right)= \left(\mathbb{S}^1, dx \otimes dx \right)$ and $\left(N, J, h\right)$ is a complete K\"ahler manifold  with bounded geometry. If the initial data $u_0\in H^k\left(\mathbb{S}^1, N\right)$ with $k\geq 2$, then the system \eqref{schf} admits a unique  flow $u\in C^0 \left(\left[0, \infty\right), H^k\left(\mathbb{S}^1, N\right) \right)$.  Particularly, when the initial data $u_0\in H^k\left(\mathbb{S}^1, N\right)$ with any $k\geq2$, then the flow $u\in C^{\infty}\left(\left[0,\infty \right)\times \mathbb{S}^1, N\right)$. 
\end{theorem}

\begin{remark}
	We also derive a  div-curl type lemma \ref{dcr} in the real line. Using this lemma to repeat our argument, the Wei-Yue Ding's conjecture for the real line case can be proved in an elementary way.
\end{remark}

For the higher dimensional case, there is also a lot of outstanding work. 
Chang-Shatah-Uhlenbeck \cite{Chang 2000} studied the raially symmetric Schr\"odinger map flow  maps from the $\mathbb{R}^2$ into  Riemann surface with small energy. Without raially symmetric assumption, Bejenaru-Ionescu-Kenig-Tataru \cite{Bejenaru 2011} obtained the global solutions for Schr\"odinger map flow  maps from the $\mathbb{R}^d$ for $d \geq 2$ into $\mathbb{S}^2$ with small energy in critical sobolev space. Besides, Bejenaru-Ionescu-Kenig-Tataru \cite{Bejenaru 2013} researched  equivariant solutions for the Schrödinger map problem from 
$\mathbb{R}^{1+2}$  to $\mathbb{S}^2$ with energy less than $4\pi$ and showed that they are global in time and scatter.  Huang-Wang-Zhao \cite{Huang 2020} considered the Schr\"odinger map flow  maps from two dimensional hyperbolic space $\mathbb{H}^2$ into sphere $\mathbb{S}^2$ with small initial data. For the Besov regularity,  the global existence for small data in critical Besov spaces was proved by  Ionescu-Kenig \cite{Ionescu 2007} and Bejenaru \cite{Bejenaru 2008} independently. Recently, Li extended
 the result of Bejenaru-Ionescu-Kenig-Tataru  \cite{Bejenaru 2011} to general K\"ahler manifolds, see \cite{Li 2018} and \cite{Li 2019}.

Next, we give the framework of this article. 
\subsection{Outline of the article.} This article is organized as follows:
In section \ref{NP},  we introduce some notations and  preliminaries, which makes the proof more easier to understand in the later section. In section \ref{bal}, we establish some balance law (i,e. $\partial_t f_1 + \partial_x f_2 = f^{re}$) associated with energy and ``momentum''. In section \ref{dcl}, we will derive a new div-curl lemma \ref{dc} in the periodic case, which palys a key role to get the regularity estimates of solution. In section \ref{55}, we note that ``div-curl'' structure is an ``inner product" in some sense, and calculate these ``inner product" induced by balance law. In section \ref{gwp}, we use the div-curl lemma \ref{dc} and continuous induction method to obatin  the priori estimates for the solution. Based on this and combining the local well-posedness results of Ding-Wang \cite{Ding 2001}, we can use  the standard interval decomposition method and extension method to obtain a global solution.

\section{Notations and Preliminaries}\label{NP}
In this section, we will make some notations and give some preliminaries.

We mark $ A \lesssim B $ to mean there exists a constant $ C > 0 $ such that $ A \leqslant C B $. We indicate dependence on parameters via subscripts, e.g. $ A \lesssim_{x} B $ indicates $ A \leqslant CB $ for some $ C = C(x) > 0 $.

To avoid ambiguity, let's make some symbolic conventions.

Suppose $\left(M,g\right)$ is a Riemann manifold with Riemann metric $g$. For any tangent vector fields $X$, $Y$ $\in TM$, we write $<X,Y>_g$ to represent the Riemann inner product for tangent vector fields $X$ and $Y$ and define $\vert X \rvert^2_g$ := $<X,X>_g$. For tangent vector fields $X$, $Y$, $Z$, $W$ $\in TM$, we define the curvature operators and curvature tensors:
\begin{align}
R(X,Y)Z:= \nabla_X \nabla_Y Z -\nabla_Y \nabla_X Z - \nabla_{\left[X,Y\right]}Z,\nonumber \\
R\left(X,Y,Z,W\right):=<R\left(Z,W\right)Y,X>_g \nonumber.
\end{align}

Under this notation convention,  the Schr\"odinger map flow maps from the real line into  K\"ahler manifolds and for maps from the circle into K\"ahler manifolds can be rewritten:
\begin{equation}\label{equ1}
J \partial_t u = \nabla_x \partial_x u.
\end{equation}.

Since we only consider that the target manifold with bounded geometry which means uniform bounds on the curvature tensor and its derivatives. So for simplicity of writing, we omit the dependence on curvature. That means we can omit the subscripts "$R$" in $ A \lesssim_{R} B $.

For convenience, let us define the following quantity:
\begin{equation*}
	m\left(t\right):=\int_{0}^{1}<\partial_x u, \partial_x u>_g 
\end{equation*}

\begin{equation*}
	E\left(t\right):=  \int_{0}^{1} <\nabla_x \partial_x u,\nabla_x \partial_x u >_g
\end{equation*}
   
\begin{equation*}
	\xi_1 \left(t\right) := \int_{0}^{t} \int_{0}^{1} det A_m
\end{equation*}

\begin{equation*}
	\xi_2 \left(t\right):= \int_{0}^{t} \int_{0}^{1} det A 
\end{equation*}   
where the $det A_m$, $det A$ are defined in section \ref{55}.

\section{law of equilibrium}\label{bal}

Let us take the inner product of both sides of \eqref{equ1} with the Riemann metric with the $\partial_t u$ and use the antisymmetry of the almost structure $J$,  we get:
\begin{equation*}
	<\nabla_x \partial_x u, \partial_t u>_g = -<\nabla_x \partial_x u, J \nabla_x \partial_x u >_g =0.
\end{equation*}

According to some simple calculations and the antisymmetry of the almost complex structure $J$, we rewrite above as follows: 
\begin{equation*}
	\partial_x <\partial_x u , \partial_t u>_g- <\partial_x u, \nabla_x \partial_t u>_g = 0,
\end{equation*}
and
\begin{equation*}
	\partial_x <\partial_x u , -J\nabla_x \partial_x u>_g- <\partial_x u, \nabla_x \partial_t u>_g = 0.
\end{equation*}

So we obtain the conservation law:
\begin{equation}\label{balance 1}
\frac{1}{2} \partial_t <\partial_x u, \partial_x u>_g -\partial_x <J\partial_x u, \nabla_x \partial_x u>_g = 0.
\end{equation}

Taking the covariant derivative of both sides of \eqref{equ1} with respect to the variable $x$, we have:
\begin{equation*}
J \nabla_x \partial_t u = \nabla_x \nabla_x \partial_x u.
\end{equation*}

According to the commutativity of covariant derivatives and differential operators, we can rewrite above as follow:
\begin{equation}\label{equ2}
J \nabla_t \partial_x u = \nabla_x \nabla_x \partial_x u.
\end{equation}

Taking the inner product of $J \partial_x u$ and $\nabla_x \partial_x u$ in the Riemann metric sense, and taking the derivative directly with respect to the variable $t$, we know: 

\begin{align}
	\partial_t <J\partial_x u, \nabla_x \partial_x u>_g&=<\nabla_x \partial_x u, J \nabla_t \partial_x u>_g + <\nabla_t\nabla_x \partial_x u , J\partial_x u>_g \nonumber\\
    &=<\nabla_x \partial_x u, J \nabla_t \partial_x u>_g+<\nabla_t\nabla_x \partial_x u, J\partial_x u>_g\nonumber\\
    &-<\nabla_x\nabla_t \partial_x u, J\partial_xu>_g+<\nabla_x\nabla_t\partial_x u, J\partial_x u>_g\nonumber\\
    &=<\nabla_x \partial_x u, J\nabla_t\partial_x u>_g+<\nabla_x\nabla_t\partial_x u, J\partial_x u>_g\nonumber\\
    &+<R\left(\partial_t u,\partial_x u\right)\partial_x u, J\partial_x u>_g\nonumber\\
    &=<\nabla_x \partial_x u, J\nabla_t\partial_x u>_g+\partial_x <\nabla_t\partial_x u, J\partial_x u>_g\nonumber\\
    &-<\nabla_t \partial_x u, J\nabla_x \partial_x u>_g+<R\left(\partial_t u,\partial_x u\right)\partial_x u, J\partial_x u>_g\nonumber\\
    &=2<\nabla_x \partial_x u, J\nabla_t\partial_x u>_g+\partial_x <\nabla_t\partial_x u, J\partial_x u>_g\nonumber\\
    &+<R\left(\partial_t u,\partial_x u\right)\partial_x u, J\partial_x u>_g\nonumber\\
    &=2<\nabla_x\partial_x u,  \nabla_x \nabla_x \partial_x u>_g-\partial_x <J\nabla_x \nabla_x \partial_x u, J\partial_x u>_g\nonumber\\
    &-<R\left(J\nabla_x\partial_x u, \partial_x u\right)\partial_x u, J\partial_x u >_g\nonumber\\
    &=2<\nabla_x\partial_x u,  \nabla_x \nabla_x \partial_x u>_g-\partial_x <\nabla_x \nabla_x \partial_x u, \partial_x u>_g\nonumber\\
    &+<R\left(\partial_x u, J \nabla_x \partial_x u\right)\partial_x u, J\partial_x u >_g\nonumber\\
    &=\partial_x<\nabla_x\partial_x u,  \nabla_x \partial_x u>_g-\partial_x <\nabla_x \nabla_x \partial_x u, \partial_x u>_g\nonumber\\
    &+\frac{1}{4}\partial_x <R\left(\partial_x u, J\partial_x u\right)\partial_x u, J\partial_x u >_g\nonumber\\
    &-\frac{1}{4} <\left(\partial_xR\right) \left(\partial_x u, J\partial_x u\right) \partial_x u, J\partial_x u>_g.
\end{align}

So we have the balance law:
\begin{align}\label{balance2}
&\partial_t <J\partial_x u, \nabla_x \partial_x u>_g -\partial_x\left(<\nabla_x \partial_x u,\nabla_x \partial_x u >_g-<\nabla_x\nabla_x \partial_x u, \partial_x u>_g\right)\nonumber\\
&-\frac{1}{4} \partial_x <R\left(\partial_x u, J\partial_x u\right) \partial_x u, J\partial_x u>_g= -\frac{1}{4} <\left(\partial_xR\right) \left(\partial_x u, J\partial_x u\right) \partial_x u, J\partial_x u>_g.
\end{align}

\section{New div-curl type lemma}\label{dcl}
In this section, we will prove a new div-curl lemma, which is useful for us to get the regularity estimates of solution.

\begin{lemma}\label{dc}
	Suppose that
	\begin{equation}
	\left\{
	\begin{aligned}
	&\partial_t f^{11} + \partial_x f^{12} =G^1\\
	& \partial_t f^{21}-\partial_x  f^{22}=G^2,
	\end{aligned}
	\right.
	\end{equation}
where $f^{i,j}$, $G^{i} $  are periodic in x with periodic 1.

	Then there hold
	\begin{align}
	&\int_{0}^{T}\int_{0}^{1} f^{11}f^{22}+f^{12}f^{21} 
	\nonumber\\
	&\lesssim
	\left(\|f^{11}\left(0\right)\|_{L^1} + \sup\limits_{0\leq t\leq T}  \|f^{11}\left(t\right)\|_{L^1} 
	 +\int_{0}^{T}\int_{0}^{1}\lvert G^1\rvert \right)\nonumber\\
	&\cdot\left(\|f^{21}\left(0\right)\|_{L^1} +\sup\limits_{0\leq t\leq T} \|f^{21}\left(t\right)\|_{L^1}  +\int_{0}^{T}\int_{0}^{1}\lvert G^2\rvert \right)\nonumber\\
	&+\int_{0}^{T} \left(\int_{0}^{1}f^{11}\int_{0}^{1}f^{22} + \int_{0}^{1}f^{12}\int_{0}^{1}f^{21}\right)
	\end{align}
provided that the right side is bounded.
\end{lemma}

\begin{proof}
We first assume that
\begin{equation}\label{as}
	\int_{0}^{1} f^{11} = 0.
\end{equation}

Noting that:	
	\begin{equation*}
	\partial_t \left(\int_{y}^{x} f^{11}  \right) +f^{12}\left(x\right)-f^{12}\left(y\right)=\int_{y}^{x} G^1.
	\end{equation*}

So we have:
\begin{equation*}
\partial_t \left(\int_{0}^{1} dy \int_{y}^{x} f^{11} \right)  + f^{12}\left(x\right) = \int_{0}^{1} f^{12} dy + \int_{0}^{1} dy\int_{y}^{x} G^1.
\end{equation*}

	Then 
	\begin{equation}\label{dd1}
	f^{21}\left(x\right)\partial_t \left(\int_{0}^{1} dy \int_{y}^{x}f^{11}\right)+f^{12}\left(x\right)f^{21}\left(x\right)= f^{21}\left(x\right) \int_{0}^{1} f^{12} dy + f^{21}\left(x\right) \int_{0}^{1} dy \int_{y}^{x} G^1
	\end{equation}
	
	\begin{equation}\label{dd2}
	\partial_t f^{21}\left( \int_{0}^{1}dy \int_{y}^{x}f^{11} \right)-\partial_xf^{22}\int_{0}^{1}dy \int_{y}^{x}f^{11}= G^2\int_{0}^{1}dy \int_{y}^{x}f^{11}
	\end{equation}
	
	\eqref{dd1}+\eqref{dd2}:

	\begin{align}
	 &\partial_t \left(\int_{0}^{1} f^{21} \int_{0}^{1}dy  \int_{y}^{x} f^{11}\right) + \int_{0}^{1} 
	\left(f^{12}f^{21}-\partial_x f^{22}\int_{0}^{1}dy \int_{y}^{x} f^{11}\right)\nonumber\\
	&=\int_{0}^{1}f^{21}\int_{0}^{1} dy\int_{y}^{x}G^1 + \int_{0}^{1} G^2 \int_{0}^{1} dy \int_{y}^{x} f^{11}  +\int_{0}^{1}f^{21}\int_{0}^{1} f^{12} dy
	\end{align}

Using \eqref{as}, and by periodic in $x$ with periodic 1, we know:
\begin{equation*}
	f^{22}\left(t,1\right)\int_{0}^{1} dy\int_{y}^{1} f^{11} - f^{22}\left(t,0\right)\int_{0}^{1} dy\int_{y}^{0} f^{11}= f^{22}\left(t,0\right)\int_{0}^{1} dy\int_{0}^{1} f^{11}=0
\end{equation*}	
	So we have:
   \begin{align}
	\int_{0}^{T}\int_{0}^{\infty} f^{11}f^{22}+f^{12}f^{21}=& \left(\int_{0}^{1} f^{21} \int_{0}^{1}dy  \int_{y}^{x} f^{21}\right) \bigg|^{T}_{0}\nonumber \\
	&+\int_{0}^{T} \int_{0}^{1}f^{21}\int_{0}^{1} dy\int_{y}^{x}G^1\nonumber\\
	& + \int_{0}^{T}\int_{0}^{1} G^2 \int_{0}^{1} dy \int_{y}^{x} f^{11}\nonumber \\
	&+ \int_{0}^{T}\int_{0}^{1}f^{21}\int_{0}^{1} f^{12}\nonumber\\
	&:= \mathcal{A}_1+\mathcal{A}_2+\mathcal{A}_3 +\mathcal{A}_4,
	\end{align}
	
where	
	
	\begin{align*}
	\lvert \mathcal{A}_1\rvert&\lesssim \|f^{11}\left(0\right)\|_{L^1}\|f^{21}\left(0\right)\|_{L^1} + \|f^{11}\left(T\right)\|_{L^1}\|f^{21}\left(T\right)\|_{L^1},
	\end{align*}
	\begin{align*}
	\lvert \mathcal{A}_2\rvert&\lesssim\int_{0}^{T} \|f^{21}\left(t\right)\|_{L^1}\|G^1\left(t\right)\|_{L^1}\\
	\lesssim&\sup\limits_{0\leq t\leq T} \| f^{21}\left(t\right)\|_{L^1}\left(\int_{0}^{T}\int_{0}^{1}\lvert G^1\rvert\right),
	\end{align*}
	\begin{align*}
	\lvert \mathcal{A}_3\rvert&\lesssim\int_{0}^{T} \|f^{11}\left(t\right)\|_{L^1}\|G^2\left(t\right)\|_{L^1}\\
	\lesssim&\sup\limits_{0\leq t\leq T} \| f^{11}\left(t\right)\|_{L^1}\left(\int_{0}^{T}\int_{0}^{1}\lvert G^2\rvert\right).
	\end{align*}

If \eqref{as} fails, we define:
    \begin{align}
	&\widetilde{f}^{11}:= f^{11}-\int_{0}^{1}f^{11}, \nonumber\\
	&\widetilde{G}^1:=G^{1}-\int_{0}^{1}G^{1}.
    \end{align}

We know:
\begin{equation}
	\partial_t \widetilde{f}^{11} + \partial_x f^{12} =\widetilde{G}^1
\end{equation}	
	Repeating the process above, we complete the proof. 
\end{proof}

\begin{lemma}\label{dcr}
	Suppose that
	\begin{equation}
	\left\{
	\begin{aligned}
	&\partial_t f^{11} + \partial_x  f^{12} =G^1,\\
	& \partial_t f^{21}-\partial_x f^{22}=G^2,
	\end{aligned}
	\right.
	\end{equation}
	
	\begin{equation}
	\begin{aligned}
	f^ {11}, f^{12}, f^{21}, f^{22} \rightarrow 0, x\rightarrow \infty.
	\end{aligned}
	\end{equation}
	Then there hold
	
	\begin{align}
	\int_{0}^{T}\int_{-\infty}^{+\infty} f^{11}f^{22}+f^{12}f^{21} &\lesssim
	\left(\|f^{11}\left(0\right)\|_{L^1} + \sup\limits_{0\leq t\leq T}  \|f^{11}\left(t\right)\|_{L^1} + \int_{0}^{T}\int_{-\infty}^{+\infty}\lvert G^1\rvert \right)\nonumber\\
	&\cdot\left(\|f^{21}\left(0\right)\|_{L^1} +\sup\limits_{0\leq t\leq T} \|f^{21}\left(t\right)\|_{L^1}  +\int_{0}^{T}\int_{-\infty}^{+\infty}\lvert G^2\rvert \right)\nonumber
	\end{align}
	provided that the right side is bounded.
\end{lemma}

\begin{proof}
   \begin{equation*}
	\partial_t \left(\int_{-\infty}^{x} f^{11} \right) +f^{12}=\int_{-\infty}^{x} G^1.
	\end{equation*}
	
	Then 
	\begin{equation}\label{dd1r}
	f^{21}\int_{-\infty}^{x} \partial_t f^{11}+f^{12}f^{21}=\int_{-\infty}^{x}G^1f^{21},
	\end{equation}
	
	\begin{equation}\label{dd2r}
	\partial_t f^{21}\int_{-\infty}^{x} f^{11}-\partial_x f^{22}\int_{-\infty}^{x}f^{11}= G^2\int_{-\infty}^{x}f^{11}.
	\end{equation}
	
	\eqref{dd1r}+\eqref{dd2r}:
	\begin{align}
	\int_{-\infty}^{+\infty} \partial_t \left(\int_{-\infty}^{x} f^{11}f^{21}\right)+&\int_{-\infty}^{+\infty}
	\left(f^{12}f^{21}-\partial_xf^{22}\int_{-\infty}^{x}f^{11}\right)\nonumber\\ 
	&=\int_{-\infty}^{+\infty}\left(f^{21}\int_{-\infty}^{x}G^1 + G^2\int_{-\infty}^{x}f^{11}\right)\nonumber.
	\end{align}
	
	We have:
	
	\begin{align}
	\int_{0}^{T}\int_{-\infty}^{+\infty} f^{11}f^{22}+f^{12}f^{21}=& \int_{-\infty}^{+\infty} \left(\int_{-\infty}^{x} f^{11}f^{21}\right)\left(0\right)- \left(\int_{-\infty}^{x} f^{11}f^{21}\right)\left(T\right) \nonumber\\
	&+\int_{0}^{T}\int_{-\infty}^{+\infty}\left(f^{21}\int_{-\infty}^{x}G^1 + G^2\int_{-\infty}^{x}f^{11}\right)\nonumber\\
	&:= \mathcal{A}_1+\mathcal{A}_2+\mathcal{A}_3,
	\end{align}
	
	where	
	
	\begin{align}
	\lvert \mathcal{A}_1\rvert&\lesssim \|f^{11}\left(0\right)\|_{L^1}\|f^{21}\left(0\right)\|_{L^1} + \|f^{11}\left(T\right)\|_{L^1}\|f^{21}\left(T\right)\|_{L^1},\nonumber
	\end{align}

	\begin{align}
	\lvert \mathcal{A}_2\rvert&\lesssim\int_{0}^{T} \|f^{21}\left(t\right)\|_{L^1}\|G^1\left(t\right)\|_{L^1}\nonumber\\
	\lesssim&\sup\limits_{0\leq t\leq T} \| f^{21}\left(t\right)\|_{L^1}\left(\int_{0}^{T}\int_{-\infty}^{+\infty}\lvert G^1\rvert\right),\nonumber
	\end{align}

	\begin{align}
	\lvert \mathcal{A}_3\rvert&\lesssim\int_{0}^{T} \|f^{11}\left(t\right)\|_{L^1}\|G^2\left(t\right)\|_{L^1}\nonumber\\
	\lesssim&\sup\limits_{0\leq t\leq T} \| f^{11}\left(t\right)\|_{L^1}\left(\int_{0}^{T}\int_{-\infty}^{+\infty}\lvert G^2\rvert\right)\nonumber.
	\end{align}

	Based analysis above, we complete the proof. 
\end{proof}

\begin{remark}\label{neiji}
	In fact, the quantity $f^{11}f^{22}+f^{12}f^{21}$ is an ``inner product" in some sense. But by coincidence, this quantity equals to the determinant of a matrix.  We  can represent the matrix in the following form
	
	\begin{equation*}
	\mathbb{A}=
	\begin{pmatrix}
	f^{11} & -f^{12} \\
	f^{21}           & f^{22}
	\end{pmatrix}.
	\end{equation*}
	
\end{remark}

\section{Some calculus from law of equilibrium}\label{55}
In this section, we will calculate some ``inner product''   induced by balance law, which paly a crucial role in the  estimates of regularity in the later sections. 

From the above remark \ref{neiji}, we can calculate some determinant of matrix to obtain some ``inner product'' .

\begin{equation}
A:=
\begin{pmatrix}
a &b \\
c          & d 
\end{pmatrix},
\end{equation}

\begin{equation*}
a:=\frac{1}{2}  <\partial_x u, \partial_x u>_g,
\end{equation*}

\begin{equation*}
b:=<J\partial_x u, \nabla_x \partial_x u>_g,
\end{equation*}

\begin{equation*}
c:=<J\partial_x u, \nabla_x \partial_x u>_g,
\end{equation*}

\begin{equation*}
d:=<\nabla_x \partial_x u,\nabla_x \partial_x u >_g-<\nabla_x\nabla_x \partial_x u, \partial_x u>_g+\frac{1}{4}<R\left(\partial_x u, J\partial_x u\right) \partial_x u, J\partial_x u>_g.
\end{equation*}

\begin{align}\label{neiji1}
det A &= ad-bc\nonumber\\
	  &=\frac{1}{2}<\partial_x u, \partial_x u>_g<\nabla_x \partial_x u,\nabla_x \partial_x u >_g-\frac{1}{2}<\partial_x u, \partial_x u><\nabla_x\nabla_x \partial_x u, \partial_x u>_g\nonumber\\
	  &+\frac{1}{8}<\partial_x u, \partial_x u>_g<R\left(\partial_x u, J\partial_x u\right) \partial_x u, J\partial_x u>_g-\lvert<J\partial_x u, \nabla_x \partial_x u>_g\rvert^2\nonumber\\
	  &= det A_m + <\partial_x u, \partial_x u>_g<\nabla_x \partial_x u,\nabla_x \partial_x u >_g-\lvert<J\partial_x u, \nabla_x \partial_x u>_g\rvert^2\nonumber\\
	  &+\frac{1}{8}<\partial_x u, \partial_x u>_g<R\left(\partial_x u, J\partial_x u\right) \partial_x u, J\partial_x u>_g\nonumber \\
	  &-\frac{1}{2} \partial_x \left(<\partial_x u, \partial_x u>_g<\nabla_x \partial_x u, \partial_x u >_g\right),
\end{align}

where \begin{equation*}
	det A_m := \lvert <\nabla_x \partial_x u, \partial_x u >_g \rvert^2.
\end{equation*}

\section{Global well-posedness for the Schr\"odinger map flow}\label{gwp}
From \eqref{balance 1}, it is easy to see $m\left(t\right)=m\left(0\right)$.

First, taking the inner product of both sides of \eqref{equ2} with the Riemann metric with the $\nabla_t \partial_x u$, we obatin:
\begin{align}\label{nlgj}
	0= <\nabla_t \partial_x u, J \nabla_t \partial_x u>_g=<\nabla_t \partial_x u, \nabla_x \nabla_x \partial_x u>_g.
\end{align}

Integrating both sides of the equation \eqref{nlgj}, and using integration by parts, according to some simple calculations, we know:
\begin{align*}
	0&=\int_{0}^{1}<\nabla_t \partial_x u, \nabla_x \nabla_x \partial_x u>_g\nonumber\\
	&=-\int_{0}^{1}<\nabla_x \nabla_t \partial_x u,  \nabla_x \partial_x u>_g+\int_{0}^{1} <\nabla_t \nabla_x \partial_x u,  \nabla_x \partial_x u>_g \nonumber\\
	&-\int_{0}^{1} <\nabla_t \nabla_x \partial_x u,  \nabla_x \partial_x u>_g\nonumber\\
	&=-\frac{1}{2}\partial_t \int_{0}^{1}<\nabla_x \partial_x u,\nabla_x \partial_x u>_g + \int_{0}^{1} <R\left(\partial_t u,\partial_x u\right)\partial_x u, \nabla_x\partial_x u>_g.
\end{align*}

Combining equation \eqref{equ1} and the antisymmetry of the curvature, we integrate by parts to get:
\begin{align}	
	&0=-\frac{1}{2}\partial_t \int_{0}^{1}<\nabla_x \partial_x u,\nabla_x \partial_x u>_g+\int_{0}^{1}<R\left(\partial_x u,  J\nabla_x \partial_x u\right)\partial_x u, \nabla_x \partial_x u >_g\nonumber\\
	&=-\frac{1}{2}\partial_t \int_{0}^{1}<\nabla_x \partial_x u,\nabla_x \partial_x u>_g+\frac{1}{2}\int_{0}^{1} <\nabla_x\left(R\left(\partial_x u, J \partial_x u\right)\partial_x u\right), \nabla_x\partial_x u >_g\nonumber\\
	&-\frac{1}{2}\int_{0}^{1} <\left(\partial_xR\right)\left(\partial_x u,  J \partial_x u\right)\partial_x u, \nabla_x \partial_x u >_g\nonumber\\
	&-\frac{1}{2}\int_{0}^{1}<R\left(\partial_x u,   J\partial_x u\right)\nabla_x\partial_x u, \nabla_x \partial_x u >_g\nonumber\\
	&=-\frac{1}{2}\partial_t \int_{0}^{1}<\nabla_x \partial_x u,\nabla_x \partial_x u>_g+\frac{1}{2}\int_{0}^{1} <\nabla_x\left(R\left(\partial_x u, J \partial_x u\right)\partial_x u\right), \nabla_x\partial_x u >_g\nonumber\\
	&-\frac{1}{2}\int_{0}^{1} <\left(\partial_xR\right)\left(\partial_x u,  J \partial_x u\right)\partial_x u, \nabla_x \partial_x u >_g
	\end{align}

According to some simple calculations, we have:
\begin{align}
	\frac{1}{2}\partial_t \int_{0}^{1}<\nabla_x \partial_x u,\nabla_x \partial_x u>_g&=-\frac{1}{2}\int_{0}^{1}<R\left(\partial_x u,  J\partial_x u\right)\partial_x u,  \nabla_x\nabla_x \partial_x u >_g\nonumber\\
	&-\frac{1}{2}\int_{0}^{1} <\left(\partial_xR\right)\left(\partial_x u,  J\partial_x u\right)\partial_x u, \nabla_x \partial_x u >_g\nonumber\\
	&=-\frac{1}{2}\int_{0}^{1}<R\left(\partial_x u,  J\partial_x u\right)\partial_x u,  J\nabla_t \partial_x u >_g\nonumber\\
	&-\frac{1}{2}\int_{0}^{1} <\left(\partial_xR\right)\left(\partial_x u,  J \partial_x u\right)\partial_x u, \nabla_x \partial_x u >_g\nonumber\\
	&=-\frac{1}{8}\partial_t \int_{0}^{1}<R\left(\partial_x u, J \partial_x u\right)\partial_x u,  J\partial_x u >_g\nonumber\\
	&+\frac{1}{8}\int_{0}^{1}<\left(\partial_t R\right)\left(\partial_x u,  J \partial_x u\right)\partial_x u, J \partial_x u >_g\nonumber\\
	&-\frac{1}{2}\int_{0}^{1} <\left(\partial_xR\right)\left(\partial_x u,  J \partial_x u\right)\partial_x u, \nabla_x \partial_x u >_g
\end{align}

We  integrate the above equation from $0$ to $t$ to obatin:
\begin{align} \label{nlds1}
	&\frac{1}{2}  E\left(t\right) + \frac{1}{8}  \int_{0}^{1} <R\left(\partial_x u, J\partial_x u\right) \partial_x u, J\partial_x u>_g\nonumber\\
	&=\frac{1}{2} E\left(0\right) + \frac{1}{8}  \int_{0}^{1} <R\left(\partial_x u, J\partial_x u\right) \partial_x u, J\partial_x u>_g \big|_{t=0}\nonumber\\
	&+\frac{1}{8}\int_{0}^{t}\int_{0}^{1} <\left(\partial_t R\right) \left(\partial_x u, J\partial_x u\right) \partial_x u, J\partial_x u>_g\nonumber\\
	&-\frac{1}{2}\int_{0}^{t}\int_{0}^{1} <\left(\partial_x R\right) \left(\partial_x u, J\partial_x u\right) \partial_x u, \nabla_x\partial_x u>_g.
\end{align}

By sobolev embedding and interpolation, we know:
\begin{align}\label{si1}
	&\int_{0}^{1} <R\left(\partial_x u, J\partial_x u\right) \partial_x u, J\partial_x u>_g\nonumber \\
	&\lesssim \int_{0}^{1} \lvert  \partial_x u\rvert^4_g\nonumber \\
	&\lesssim \|\nabla_x \partial_x u\|_{L_g^2} \|\partial_x u\|^3_{L_g^2} + \|\partial_x u\|^4_{L_g^2}\nonumber\\
	&\lesssim E^{1/2}\left(t\right) m^{3/2}\left(0\right) + m^2\left(0\right), 
\end{align}

\begin{align}\label{si2}
&\int_{0}^{t}\int_{0}^{1} <\left(\partial_t R\right) \left(\partial_x u, J\partial_x u\right) \partial_x u, J\partial_x u>_g\nonumber\\
& \lesssim \int_{0}^{t}\int_{0}^{1} \lvert \nabla_x \partial_x u\rvert_g \lvert  \partial_x u\rvert^4_g\nonumber\\
&\lesssim \int_{0}^{t} \|\lvert\nabla_x \partial_x u\rvert_g\|_{L^2} \| \lvert  \partial_x u\rvert^4_g \|_{L^2}\nonumber\\
&\lesssim \int_{0}^{t}  \|\lvert\nabla_x \partial_x u\rvert_g\|_{L^2}  \| \lvert  \partial_x u\rvert^2_g \|^2_{L^4}\nonumber\\
&\lesssim \int_{0}^{t}  \|\lvert\nabla_x \partial_x u\rvert_g\|_{L^2}\|\partial_x \lvert \partial_x u\rvert^2_g\|_{L^2}\| \lvert \partial_x u\rvert^2_g\|_{L^1}+\|\lvert\nabla_x \partial_x u\rvert_g\|_{L^2}\| \lvert \partial_x u\rvert^2_g\|^2_{L^1}\nonumber\\
&\lesssim 2 \sup\limits_{t} E^{1/2}\left(t\right)m\left(0\right)\xi^{1/2}_1\left(t\right)t^{1/2} + \sup\limits_{t} E^{1/2}\left(t\right)m^2\left(0\right)t,
\end{align}

\begin{align}\label{si3}
&\int_{0}^{t}\int_{0}^{1} <\left(\partial_x R\right) \left(\partial_x u, J\partial_x u\right) \partial_x u, \nabla_x\partial_x u>_g\nonumber\\
& \lesssim \int_{0}^{t}\int_{0}^{1} \lvert \nabla_x \partial_x u\rvert_g \lvert  \partial_x u\rvert^4_g\nonumber\\
&\lesssim 2 \sup\limits_{t} E^{1/2}\left(t\right)m\left(0\right)\xi^{1/2}_1\left(t\right)t^{1/2}+ \sup\limits_{t} E^{1/2}\left(t\right)m^2\left(0\right)t.
\end{align}

Combining \eqref{nlds1}, \eqref{si1}, \eqref{si2},  and \eqref{si3}, we get:
\begin{align}\label{bts1}
	\frac{1}{4} E\left(t\right) &\lesssim \frac{1}{2} E\left(0\right) + m^3\left(0\right) + m^2\left(0\right)\nonumber\\
	 &+ \frac{3}{4}\left(2 \sup\limits_{t} E^{1/2}\left(t\right)m\left(0\right)\xi^{1/2}_1\left(t\right)t^{1/2}+ \sup\limits_{t} E^{1/2}\left(t\right)m^2\left(0\right)t\right)
\end{align}

According to the div-curl lemma \ref{dc}, we have:
\begin{align}\label{dcs}
\xi_2\left(t\right)&\lesssim  \int_{0}^{t} \int_{0}^{1}<\partial_x u, \partial_x u>_g\int_{0}^{1}<\nabla_x \partial_x u,\nabla_x \partial_x u >_g\nonumber\\
&-\int_{0}^{t} \int_{0}^{1}<\partial_x u, \partial_x u>_g\int_{0}^{1} <\nabla_x\nabla_x \partial_x u, \partial_x u>_g\nonumber\\
&+\frac{1}{4}\int_{0}^{t} \int_{0}^{1}<\partial_x u, \partial_x u>_g\int_{0}^{1}<R\left(\partial_x u, J\partial_x u\right) \partial_x u, J\partial_x u>_g\nonumber\\
&+\sup\limits_{t}\|<\partial_x u,\partial_x u>_g\|_{L^1}\sup\limits_{t}\|<\nabla_x \partial_x u,J \partial_x u>_g\|_{L^1}\nonumber\\
&+\frac{1}{4}\sup\limits_{t}\|<\partial_x u,\partial_x u>_g\|_{L^1}\int_{0}^{t}\int_{0}^{1} \lvert <\left(\partial_x R\right) \left(\partial_x u, J\partial_x u\right) \partial_x u, J\partial_x u>_g\rvert\nonumber\\
&\lesssim 2\int_{0}^{t} \int_{0}^{1}<\partial_x u, \partial_x u>_g\int_{0}^{1}<\nabla_x \partial_x u,\nabla_x \partial_x u >_g\nonumber\\
&+\frac{1}{4}\int_{0}^{t} \int_{0}^{1}<\partial_x u, \partial_x u>_g\int_{0}^{1}<R\left(\partial_x u, J\partial_x u\right) \partial_x u, J\partial_x u>_g\nonumber\\
&+m\left(0\right)\left(\sup\limits_{ t}E^{1/2}\left(t\right)m^{1/2}\left(0\right) +\frac{1}{4}\int_{0}^{t}\int_{0}^{1} \lvert <\left(\partial_x R\right) \left(\partial_x u, J\partial_x u\right) \partial_x u, J\partial_x u>_g\rvert \right)
\end{align}

By sobolev embedding and interpolation, we know:
\begin{align}\label{dcs1}
&2\int_{0}^{t} \int_{0}^{1}<\partial_x u, \partial_x u>_g\int_{0}^{1}<\nabla_x \partial_x u,\nabla_x \partial_x u >_g\nonumber\\
&\leq 2m\left(0\right)\sup\limits_{t} E\left(t\right)t
\end{align}

\begin{align}\label{dcs2}
&\frac{1}{4}\int_{0}^{t} \int_{0}^{1}<\partial_x u, \partial_x u>_g\int_{0}^{1}<R\left(\partial_x u, J\partial_x u\right) \partial_x u, J\partial_x u>_g\nonumber\\
&\lesssim \frac{1}{4} m\left(0\right) \int_{0}^{t}\int_{0}^{1} \lvert \partial_x u\rvert^4_g \nonumber \\   
&\lesssim \frac{1}{4} m\left(0\right) \int_{0}^{t}\| \partial_x \lvert\partial_x u \rvert^2_g\|^{2/3}_{L^2} \|\lvert \partial_x u\rvert^2_g\|^{4/3}_{L^1} + \|\lvert \partial_x u\rvert^2_g\|^{2}_{L^1}\nonumber\\
&\lesssim \frac{1}{2} m^{7/3}\left(0\right)\xi^{1/3}_1\left(t\right)t^{2/3}+\frac{1}{4}m^3\left(0\right)t
\end{align}

\begin{align}\label{dcs3}
&\frac{1}{4}\int_{0}^{t}\int_{0}^{1} \lvert <\left(\partial_x R\right) \left(\partial_x u, J\partial_x u\right) \partial_x u, J\partial_x u>_g\rvert\nonumber\\ 
&\lesssim\frac{1}{4}\int_{0}^{t}\int_{0}^{1} \lvert \partial_x u\rvert^5_g\nonumber\\
&\lesssim\frac{1}{4}\int_{0}^{t}\| \partial_x \lvert\partial_x u \rvert^2_g\|_{L^2} \|\lvert \partial_x u\rvert^2_g\|^{3/2}_{L^1} + \|\lvert \partial_x u\rvert^2_g\|^{5/2}_{L^1}\nonumber\\
&\lesssim \frac{1}{4}m^{5/2}\left(0\right)t + \frac{1}{2}m^{3/2}\left(0\right)\xi^{1/2}_1\left(t\right)t^{1/2}	
\end{align}

Besides, from \eqref{neiji1}, we get:
\begin{equation}\label{dy1}
	\xi_1\left(t\right) \leq \xi_2\left(t\right) + \frac{1}{8} \int_{0}^{t}\int_{0}^{1} \lvert<\partial_x u, \partial_x u>_g<R\left(\partial_x u, J\partial_x u\right) \partial_x u, J\partial_x u>_g\rvert
\end{equation}

Similarly, we obatin:
\begin{align}\label{dcs4}
&\frac{1}{8} \int_{0}^{t}\int_{0}^{1} \lvert<\partial_x u, \partial_x u>_g<R\left(\partial_x u, J\partial_x u\right) \partial_x u, J\partial_x u>_g\rvert\nonumber\\
&\lesssim\frac{1}{8}\int_{0}^{t}\int_{0}^{1} \lvert \partial_x u\rvert^6_g\nonumber\\
&\lesssim\frac{1}{8}\int_{0}^{t} \| \partial_x \lvert\partial_x u \rvert^2_g\|^{4/3}_{L^2} \|\lvert \partial_x u\rvert^2_g\|^{5/3}_{L^1} + \|\lvert \partial_x u\rvert^2_g\|^{3}_{L^1}\nonumber\\
&\lesssim\frac{1}{8}m^3\left(0\right)t + \frac{1}{4}m^{5/3}\left(0\right)\xi^{2/3}_1\left(t\right)t^{1/3}
\end{align}

Combining \eqref{dcs}, \eqref{dcs1}, \eqref{dcs2}, \eqref{dcs3}, \eqref{dy1}, and \eqref{dcs4}, we have:
\begin{align}\label{bts2}
\xi_1\left(t\right)&\lesssim \frac{1}{8}m^3\left(0\right)t + \frac{1}{4}m^{5/3}\left(0\right)\xi^{2/3}_1\left(t\right)t^{1/3}+2m\left(0\right)\sup\limits_{t} E\left(t\right)t\nonumber\\
&+\frac{1}{2} m^{7/3}\left(0\right)\xi^{1/3}_1\left(t\right)t^{2/3}+\frac{1}{4}m^3\left(0\right)t\nonumber\\
&+m\left(0\right)\left(\sup\limits_{ t}E^{1/2}\left(t\right)m^{1/2}\left(0\right)+\frac{1}{4}m^{5/2}\left(0\right)t + \frac{1}{2}m^{3/2}\left(0\right)\xi^{1/2}_1\left(t\right)t^{1/2}\right)
\end{align}

Combining \eqref{bts1}, \eqref{bts2},  and using the continuous induction method and  standard bootstrap arguments, we know there exists a $T^1=T^1\left(m\left(0\right)\right)$, and when $t\in \left[0, T^1\right)$, we have:
\begin{equation}
	\begin{aligned}
    E\left(t\right) \leq C\left(m\left(0\right)\right)\left(E\left(0\right)+1\right).
	\end{aligned}
\end{equation}

From the local well-posedness results of Ding-Wang \cite{Ding 2001}, we can use  the standard interval decomposition method and extension method to obtain a global solution $u \in C^0 \left(\left[0, \infty \right), H^2\left(\mathbb{S}^1, N\right)\right)$.

\par{\bf Acknowledgement}:

This work was supported by the National Natural Science Foundation of China (No. 12171097), the Key Laboratory of Mathematics for Nonlinear Sciences (Fudan University), Ministry of Education of China, P.R.China. Shanghai Key Laboratory for Contemporary Applied Mathematics, School of Mathematical Sciences, Fudan University, P.R. China, and by Shanghai Science and Technology Program [Project No. 21JC1400600].

\par{\bf Data Availability}:

The manuscript has no associated data.

\par{\bf Declarations}:

\par{\bf Conflicts of interest}:

The authors state that there is no conflict of interest.

\end{document}